\documentclass[12pt]{amsart}
\usepackage{latexsym,fancyhdr,amssymb,color,amsmath,amsthm,graphicx,listings,comment}
\usepackage[section]{placeins}
\newtheorem{thm}{Theorem} 
\newtheorem{dfn}{Definition} 
 
\newtheorem{lemma}{Lemma} 
\newtheorem{coro}{Corollary}

\let\paragraph\subsection
\title{Calculus on Wave Fronts}
\author{Oliver Knill}
\date{March 22, 2026}
\address{Department of Mathematics \\ Harvard University \\ Cambridge, MA, 02138 }
\subjclass{}

\begin{document}

\begin{abstract}
We define a deformation of the exterior derivative that
is a bounded operator and preserves the symmetries of the geometry. 
It satisfies a modified wave equation that honors the strong Huygens
principle in all dimensions. 
\end{abstract}
\maketitle

\section{Deformed exterior derivative}

\paragraph{}
Denote by $\mathcal{H}=\oplus_{k=1}^q \overline{\Lambda}^k(M)$ 
the Hilbert space of {\bf differential forms} on a Riemannian q-manifold $(M,g)$. 
It is the closure of the space of compactly supported 
differential forms $\Lambda(M)$ with respect to the norm defined by the inner product 
$\int_M \langle f,g \rangle_p dx$, where $dx$ is the volume $q$-form on $M$. 
(See section 11.3 in \cite{Cycon}).
With the {\bf exterior derivative} $d$ and its {\bf adjoint} $d^*$, 
define the {\bf Dirac operator} $D=d+d^*$ and the 
{\bf Hodge Laplacian} $L=D^2=d^* d + d d^*$. The wave equation $u_{tt} = - L u$
has the explicit solution $u(t,x) = \cos(D t) u(0,x) + t {\rm sinc}(D t) u_t(u,x)$ in 
any dimension. It defines a global Hamiltonian flow in $\mathcal{H} \times \mathcal{H}$. 
The wave equation is equivalent to two Schr\"odinger equations $u_t=\pm i D u$
with a 2-dimensional solution space $c_+ e^{i D t} u + c_- e^{-i D t} v$ in $\mathcal{H}$. 
\footnote{When restricting to functions (0-forms), pseudo differential operators are involved.
Extending the wave equation to all forms allows to avoid pseudo differential 
operators like $\sqrt{-\Delta}$. While pseudo differential operators are in general 
non-local, the expressions $\cos(Dt),{\rm sinc}(Dt)$ are both even and so functions
of $L$. The solution formula is the same than for a wave equation on a discrete geometry
like a finite abstract simplicial complex. }

\paragraph{}
We work here on Euclidean space $\mathbb{R}^q$. But all notions that can be pushed over via
coordinate changes hold also in general. In the case of a Riemannian manifold $M$, we would
use the {\bf exponential map} $\exp_p: T_pM \to M$ to transport the unit sphere
$S_1(p)= \{ v \in \mathbb{R}^q, |v|=1\} \subset T_pM$ to the {\bf wave front}
$W_h(p) = \exp_p(h S_1(p)) \subset M$. For $h$ smaller than the {\bf radius of injectivity},
the wave front is diffeomorphic to a $(q-1)$-sphere. We can look at partial differential 
equations like the wave equation also on Riemannian manifolds. We work here for simplicity 
on flat $M=\mathbb{R}^q$.

\paragraph{}
The mathematics of wave fronts links differential geometry 
with partial differential equations. \cite{BergerPanorama} outlines many of these relations. 
First investigations were done by Huygens who studied with waves and rays 
\cite{HuygensBarrowNewtonHooke}. The principle that in odd dimensions, the solution 
$u(t,p)$ of the wave equation $u_{tt}=-Lu$ only involves the initial data on $W_h(p)$ is 
called the {\bf strong Huygens principle}. It implies that there are {\bf sharp wave fronts}. 
In even dimensions, also points closer to $p$ do matter. Wave fronts are softer;
there are {\bf wakes}. The Kirchhoff solutions of the wave equation for scalar fields
can be found for example in section 2.4 in \cite{Evans2010}. 

\paragraph{}
Let $\phi_n(t)$ denote the unique solution of {\bf Bessel equation}
\begin{equation} f''(r) + (n-1) \frac{f'(r)}{r} + f(r)=0, f(0)=1,f'(0)=0 \; . \end{equation}
We have $\phi_1(r)=\cos(r), \phi_2(t)=J_0(r)$, $\phi_3(t)={\rm sinc}(r)=\sin(r)/r$,
$\phi_4(r)=2J_1(r)/r$. Historically, Euler was one of the first to look in 1764 at the $q=2$ 
wave equation $u_{tt} = u_{rr} + \frac{1}{r} u_r + \frac{1}{r^2} u_{\theta \theta}$ 
which with $u(t,r,\theta)=f(r) \cos(m t)$ leads to $f''+ f'/r+ m^2 f=0$ and which 
after a variable change $r \to mr$ becomes $g'' + g'/r + g=0$, the above Bessel equation
in the case $q=2$.  (See \cite{Watson1962} section 1.3). In an appendix, we say more about
Bessel. 

\paragraph{} The following definition is new: 

\begin{dfn}
For $t>0$ define the {\bf deformed exterior 
derivative} using the $\phi_{q+2}$ Bessel function:
\begin{equation} d_t = t \phi_{q+2}(t D) d  \; . \end{equation}
\end{dfn}

\paragraph{}
Because $\phi_{q+2}(0)= 1$ and $d_0=0$, 
it satisfies $\lim_{t \to 0} d_t/t = \lim_{t \to 0} \frac{d}{dt} d_t  =d$. 
The operator $d_t$ is a bounded operator on $\mathcal{H}$ because 
$r \to \phi_{q+2}(r) r$ is bounded. By the functional calculus for
the self-adjoint operator $tD$ (which is only densely defined), also 
$\phi_{q+2}(tD) tD$ is bounded so that $\phi_{q+2}(tD) t d$ is bounded. 
A densely defined bounded operator extends to the entire Hilbert space. 

\paragraph{}
For $q=1$, where $\phi_3(r)={\rm sinc}(r)=\sin(r)/r$ and 
$D=d+d^*=\left[ \begin{array}{cc} 0 & -d/dx \\ d/dx & 0 \end{array} \right]
 = i d/dx$, we have $t \; {\rm sinc}(D t) D u = \sin(Dt) u = [e^{i Dt} - e^{-i Dt}] u/(2i)
= [e^{-\frac{d}{dx} t} - e^{\frac{d}{dx} t}] u/(2i) = i [u(x+t)-u(x-t)]/2$, using the {\bf Taylor
theorem} $e^{\frac{d}{dx} t} u(x)=\sum_{n=0}^{\infty} \frac{d^n}{dx^n} \frac{u(x)}{n!}=u(x+t)$.
We see that on 0-forms $u$, its exterior derivative is the 1-form 
$d_t u(x) = [u(x+t) - u(x-t)]/2 dx$. On a circle $M$, the only 
compact 1-manifold, all $t$ rationally independent
of the length of $M$ have the same harmonic forms $d_t u =0$. The Betti numbers $b_k$
defined as the dimension of the kernel of $L_t=d_t^* d_t + d_t d_t^*$ restricted to $\mathcal{H}_k$ 
are for almost all $t$ the same as for the $L=d d^* + d^* d$. 

\paragraph{}
Our original goal (pursued sporadically over the last two decades) 
had been to define an exterior derivative in arbitrary dimensions that preserves 
symmetries and does not require smoothness. It is the quest to pursue 
{\bf multi-variable calculus without limits}. 
A concrete early question has been: is it true that if we take a 1-form
$f$ (vector field) on a 2-manifold and have all line integrals along all $p \in M$ and fixed $t>0$
the wave fronts $W_t(p)$ are all  zero for a fixed $t$, that $f$ has to have zero classical curl?
We would call the line integral $\int_{W_t(p)} f$ a {\bf discretized curl} as it does not mind $f$
to be only continuous or even piecewise continuous and allow $M$ to be piecewise smooth only, like a 
{\bf polyhedron} or have boundaries, where $W_t(p)$ is defined using the billiard flow. 
Green's theorem tells that the line integral along $W_t(p)$ is the integral of the curl $df$ 
over the ball $B_t(p)$ (for small $t$ at least). We can now answer this question and see that it 
depends on the spectrum of the form Laplacian on $M$ as well as on $t$. 

\paragraph{}
Any discretization in the form of a cell complex or triangulation 
would destroy rotational symmetry for example. Rotational symmetry is important in physics. 
The properties of the {\bf hydrogen operator} for example explains much of the periodic elements in 
chemistry, the properties of Schwarzschild or Kerr metric explain measurable consequences in the
general theory of relativity. In one dimensions, the scaled deformed derivative $d_t/t$ is the 
standard symmetric difference gradient. Since $|t \phi_{q+2}(t)| \sim t^{(1-q)/2}$ for $t \to \infty$,
we have $||d_t|| \sim t^{(1-q)/2}$ for $t \to \infty$. Only in dimension $1$, the exterior
derivative does not go to zero and stays bounded for $t \to \infty$. 

\section{Deformed Wave equation}

\paragraph{}

\begin{dfn}
Define the {\bf Bessel acceleration} 
\begin{equation}  D_{tt} u = u_{tt} + (q-1) [ \frac{u_t}{t} - \frac{u}{t^2} ] \; . \end{equation}
\end{dfn}
It needs that $u(0,x)=0$ to have a limit at $t \to 0$. 

\paragraph{}
It is a linear second order {\bf Sturm-Liouville operator} that is singular 
at $t=0$ and approaches $u_{tt}$ for larger $t$.  It factors as
$D_{tt}=(\partial_t+\frac{q}{t})(\partial_t-\frac{1}{t})$. It is a 
Helmholtz Laplacian with zero energy and angular momentum $1$.

\paragraph{}

\begin{dfn}
Define also the {\bf Bessel acceleration} 
\begin{equation}  d_{tt} u = u_{tt} + (q-1)  \frac{u_t}{t} \; . \end{equation}
\end{dfn}
It needs that $u_t(0,x)=0$ to have a limit at $t \to 0$.

\paragraph{}
The differential equation $D_{tt}=a$ has for $u(0)=0$ the solution family 
$u(t) = a t^2/(q+1) + t u'(0)$. This agrees with solutions of 
$\partial_t^2 u=0$ except that the acceleration
is for $q>1$ smaller than the traditional acceleration $u_{tt}$. 
The equation $D_{tt} f = g$ is solved for analytic 
$g=\sum_{n=1}^{\infty} c_n t^{n-1}$ by noting that $D_{tt} f = t^{n-1}$ has
the solution $-\frac{t (1-t^n)}{n (q + n)}$. If $g=\sum_{n=1}^{\infty} a_n t^{n-1}$ 
is bounded, there is exactly one bounded solution $f = \sum_{n=1}^{\infty} b_n t^n$.
The condition $f(0)=0$ is one of boundary condition, 
the boundedness forces the initial velocity and so fixes a second boundary condition.

\paragraph{} Also this definition of a partial differential equation appears to be new:

\begin{dfn}
Define the {\bf deformed wave equation}
\begin{equation} D_{tt} u + L u = 0 \; , \end{equation}
and
\begin{equation} d_{tt} u + L u = 0 \; , \end{equation}
\end{dfn}

\paragraph{}
For $q=1$, these are  both the usual wave equation. The motivation to define thes PDE's
was that we will see that $u(t,p) = d_t f(p) = t \phi_{q+2}(tD) df$ is an explicit solution to 
$(D_{tt} + L)u=0$  with $u(0,p)=0$ and $\lim_{t \to 0} \frac{d}{dt} d_t f(p) = df(p)$.
Also, $u(t,p) =\phi_q(tD) df(p)$ is a an explict solution of $(d_{tt} + L) u=0$ with 
$u(0,p)=df(p)$ and $u_t(0,p)=0$. 

\paragraph{}
In order that the solution of the deformed wave equation $(D_{tt} + L) u=0$ works we need 
in the case $q \neq 1$ that $u(t,x) = t^2 v(t,x)$ and especially $u(0,x)=0$. 
In order that the solution of the deformed wave equation $(d_{tt} + L) u =0$ makes sense,
we need $u_t(0,x)=0$. 
In all dimension $q$, the initial value problems with initial velocity or initial positions 
can be solved explicitly. The strong Huygens principle holds. 

\paragraph{}
Asymptotically, for $t \to \infty$, for any $q$, the solutions of the deformed wave equation 
move like the wave equation $u_{tt} + L u =0$. While the later has the explicit solution 
$u(t,p) = \cos(Dt) u(0,p) + {\rm sinc}(Dt) u'(0,p)$, this only honors the {\bf weak Huygens
property}: the classical wave equation features ``wakes" in even dimensions as the explicit
solution formulas of Kirchhoff show.  By modifying the time acceleration from $\partial_t^2$ to 
$D_{tt}$ slightly, the strong Huygens principle is achieved in arbitrary dimensions. This
can be useful when trying to interpret a wave front in a space-time 4-manifold  as ``space". 
\footnote{As for references, see e.g. \cite{Simon2017} 6.9.64 which 
uses the Fourier picture that works in arbitrary dimensions.
Or \cite{Zeldich2017} Chapter I, section 3, for wave equations 
on Riemannian manifolds.}

\section{Results}

\paragraph{}
Here are our main results about the bounded operator $d_t: \mathcal{H} \to \mathcal{H}$. 
Similarly as the classical exterior derivative $d$, also $d_t$ maps 
$k$-forms $\mathcal{H}_k$ to $(k+1)$-forms $\mathcal{H}_{k+1}$ but it is a bounded operator.
From $d^2=0$ follows immediately that also $d_h^2=0$. Harmonic forms stay harmonic. 

\begin{thm}[Wave equation]
a) The form $u(t,p)=d_t f(p)$ solves $(D_{tt} + L) =0$ 
with initial position $u(0,p)=0$ and initial velocity $\lim_{t \to 0} \frac{d}{dt} u_t(0,p) = df(p)$.\\
b) The form $u(t,p)=\phi_q(tD) df(p)$ solves $(d_{tt} + L)=0$ 
with initial position $u(0,p)=df(p)$ and initial velocity $u_t(0,p)=0$. 
\end{thm}
\begin{proof}
a) Just verify that $u(t,x) = t \phi_{q+2}(tD) Y(x)$
solves $u_{tt} + u_t (q-1)/t + u (q-1)/t^2 + D^2 u = 0$ provided that
$\phi''(r) + (q+1) \phi'(r)/r  + \phi(r) = 0$ with $f(0)=1$, $f'(0)=0$.
For $r=tD$ we get $\phi''(tD) + (q+1) \phi'(tD)/(tD) + \phi(tD) = 0$.
Differentiate $u$ twice using chain and product rule:
\begin{eqnarray*}
u    &=&  t \phi  (tD)     Y                                         \\
u_t  &=&  t \phi' (tD) D   Y                    +   \phi (tD)   Y    \\
u_tt &=&  t \phi''(tD) D^2 Y  +  \phi'(tD) D Y  +   \phi'(tD) D Y   \; . 
\end{eqnarray*}
Multiply the first by $D^2$, the second by $(q-1)/t$ and switch $q-1$ 
to $q+1$ using
$t \phi'(tD) D^2 Y (q-1)/(tD)  = t \phi'(tD) D^2 Y (q+1)/(tD) - 2 \phi'(tD) D Y$.  
\begin{eqnarray*}
D^2 u        &=& t \phi  (tD)  D^2 Y                                              \\
u_t (q-1)/t  &=& t \phi' (tD)  D^2 Y (q+1)/(tD) +  \phi(tD)    Y (q-1)/t - 2\phi'(tD) D Y \\
u_tt         &=& t \phi''(tD)  D^2 Y            + 2\phi'(tD) D Y                      \\
-u (q-1)/t^2 &=&                                -  \phi(tD)    Y (q-1)/t \; . 
\end{eqnarray*}
The sum of the terms on the left hand side add up to 0 as $\phi$ satisfies the 
deformed wave equation. The right hand side add up also to zero
because $\phi$ satisfies the Bessel equation $\phi''+(q+1) \phi'/(tD) + \phi=0$.  
b) By taking $q$ rather than $q+2$, the forms $u(t,p)=\phi_q(Dt) f(p)$ 
have solve the {\bf zero angular momentum PDE}
$$ u_{tt} + (q-1) \frac{u_t}{t}  = -D^2 u  $$
with $u(0,p)=df$ and $u_t(0,p)=0$. 
The verification is the same. 
 \end{proof} 

\paragraph{}
Because the functions $\phi_k$ are all bounded, $d_t$ is a bounded operator 
on the Hilbert space $\mathcal{H}$ of forms, for all $t \geq 0$. (The case $t=0$
holds trivially because $d_t=0$ for $t=0$.)

\paragraph{}
On other original motivation was to 
search for modified wave equations that solve
$$ \phi_n(Dt) u(0,p) + t \phi_{n+2}(Dt) u_t(0,p) $$ 
because $u(t,p) = \phi_1(Dt) u(0,p) + t \phi_3(Dt) u_t(0,p)$ satisfies the 
classical wave equation $u_{tt} + D^2 u = 0$ in all dimensions $q \geq 1$,
as $\phi_1(r)=\cos(r), \phi_3(r)={\rm sin}(r)/r={\rm sinc}(r)$. 
It turned out that we can indeed find such equations but that in the 
position and velocity case, different differential equations appeared. 

\paragraph{}
As we have seen in the introduction,
the $1$-dimensional case $q=1$ is a prototype case. It is historically important as d'Alembert
was a pioneer in partial differential equations. Before him, wave phenomena had been discussed
in a descriptive way. The deformed wave equation agrees in the case $q=1$ with the wave equation 
$u_{tt}+Lu=0$. In one dimensions, $u(t,p)=t {\rm sinc}(Dt) f$ satisfies the wave equation with 
$u(0)=0,u_t(0)=df$. It is the {\bf d'Alembert solution}
$u(t,p) = [u(0,p+t)-u(0,p-t)]/2$. The expression $u(t,p)/t$ is a {\bf discrete numerical derivative}
that converges for $t \to 0$ to the usual derivative $df$. In the limit $t \to 0$, we need
the initial $0$-form $f$ to be differentiable. The calculus for $d$ needs limits, the calculus
for $d_t$ does not need any limits for $t>0$. 

\paragraph{ }
The second main result is about the modified partial differential equation: 

\begin{thm}[Strong Huygens] 
a) $u(t,p) = d_t f(p) = t \phi_{q+2}(tD) df$ only uses $f$ on $W_t(p)$. \\
b) $u(t,p) = \phi_q(tD) df$ only uses $f$ on $W_t(p)$. 
\end{thm}

\begin{proof} 
We focus on a) and give in the appendix the argument to reduce b) to a) using the 
identity $\phi_q(r) = r^{1-q} (\phi_{q+2}(r) r^q)'/q$.  \\
(i) The statement is first first shown for $(q-1)$-forms $f$, where it reduces to the
{\bf ball average formula} applied to the $q$-form $g=df$. 
There is a Taylor expansion \cite{Ovall2014} for the ball
expectation $|B_t|^{-1}\int_{B_t} g = \sum_n c_n t^{2n} L^n g$. We noticed that this matches with 
$\phi_{q+2}(tD) g$. {\bf Stokes theorem} then shows that $d_t f$ only invokes $f$ on the wave
front $W_t(p)$.  Writing $\frac{1}{|S_t|} \int_{S_t(p)} f(x) dS$ as 
$\sum_{n=0}^{\infty} c_n t^{2n} L^n f(p)$
is a {\bf Pizzetti type formula}, named after 
{\bf Paolo Pizzetti} \cite{pizzetti1909} who derived 
such a formula first in dimension $q=3$.  \\
(ii) To see the general case for $k$-forms with $k<q-1$, 
we use {\bf inner derivatives} $i_X$ and induction $k \to k-1$ starting
with the induction assumption $k=q-1$, seen in (i). 
If $X$ is a vector field, denote by $i_X: \Lambda_{k+1} \to \Lambda_{k}$ 
the {\bf inner derivative} and by $L_X = i_X d + d i_X$ 
the {\bf Lie derivative} preserving $k$-forms. For background, see \cite{AMR}. 
Expressing the Lie derivative as an anti-commutator of exterior and interior 
derivative is known as {\bf Cartan's magic formula}. In order to extend the 
strong Huygens property from $k$-forms to to $(k-1)$-forms, 
use a $k$-form $f$ and a constant vector field for which $L_X g=0$ so that $i_X d f = - d i_X f$
and hence $i_X d_h f = - d_h i_X f$. The Huygens principle so holds for any 
$(k-1)$-form $i_X f$ provided $f$ was $L_X$ invariant. \\
(iii) In order that (ii) works for all $k$-forms, we need to show for all $k \leq q-1$, 
any $k$-form $f$ can be written as a linear combination of $k$-forms that are in the 
kernel of some $L_X$. 
This however follows from the {\bf polarization identity}. It allows to write 
any monomial $\prod_i x_i^{n_i}$ of degree $n=\sum_j n_j$
as a linear combination of terms $(\sum_i s_i x_i)^n$, where $s_i \in \{-1,1\}$.
Every differential $k$-form $a df$ can be written as a combination of forms 
$i_X g$, where $g$ is a $k$-form such that $L_X g=0$. 
\end{proof} 

\paragraph{}
Lets look at the case $q=3$ and $k=2$. We want to show that we can get any $1$-form 
as a linear combination of forms $i_X g$, where $g$ is a $2$-form and $L_X g = 0$. 
For example, the $2$-form $g=(x+y-z)^n d(x-y) \wedge dy$ satisfies 
$L_X g=0$ for $X=(1,-1,0)$ and $I_X g = (x+y-z)^n dy$.
Using polarization, any monomial 1-form $x^k y^l z^m dy$ can be realized in such a way
Again by linearity, every polynomial 1-form $Q(x,y,z) dy$ has this property.
Using linear combinations of such forms, we can realize any $1$-form
$P(x,y,z) dx + Q(x,y,z) dy + R(x,y,z) dz$ as a linear combination of 
inner derivatives of $1$-forms $i_X g$, where $g$ is a 2-form satisfying $L_X g = 0$. 

\paragraph{}
In the proof, we used the ball average formula in Euclidean space and 
that in flat Euclidean space we can use vector fields $X$ that are
parallel. In a curved space, there are no parallel vector fields in general. 
Also ball averages are different. But once we have established the strong 
Huygens principle for $D_{tt} u = - D^2 u$ in Euclidean space
$\mathbb{R}^n$ we have, (like in the usual 
wave equation $u_{tt} = - D^2 u$), that the strong Huygens property goes 
over to the Riemannian manifold case, using that the exponential map 
near a point is just a coordinate change and that $D=d+d^*$ is a coordinate
independent notion. 

\section{Harmonic forms}

\paragraph{}
Define $D_h = d_h + d_h^*$ and $L_h=D_h=d_h^* d_h + d_h d_h^*$ so that with 
$\psi_n(r)=\phi_n(r) r$, we have $D_h=\psi_{q+2}(h D)$. 
Since $\psi_{q+2}(0)=0$, all classically harmonic forms $f$ are also 
{\bf $h$-harmonic}: they satisfy $L_h f=0$. 
For any compact Riemannian manifold $M$ there is a Hilbert space $\mathcal{H}$ and a Dirac operator 
$D=d+d^*$. The set-up does not require coordinates and locally, we can parametrize a coordinate 
patch and so have the above results for $h$ smaller than the {\bf radius of injectivity} $\rho(M)$ 
of $M$. In the compact space $D$ has discrete spectrum so that $\psi_h(D)$

\begin{thm}[Harmonic forms]
If $M$ is a compact Riemannian manifold, then there are no new $h$-harmonic forms 
for almost all $h<\rho(M)$. 
\end{thm}

\paragraph{}
For example, if $M=\mathbb{T}^1$ with a metric giving it length is $1$, 
then for all irrational $0<h<1/2$, no new h-harmonic forms are obtained for $d_h$. 
For rational $h>0$, all $2h$-periodic function $f$ or 1-form is h-harmonic because 
$2d_h f(p)=f(p+h)-f(p-h)=0$. 

\paragraph{}
If cohomology is defined via Hodge, the {\bf Betti numbers} $b_k$ can be defined as 
the dimension of the space of $h$-harmonic $k$-forms.

\begin{coro}
The Hodge cohomology does not change for most $h$. 
\end{coro}

\section{Physics}

\paragraph{}
Our goal had been to generalize the 1-dimensional discrete derivative to higher
dimensions without introducing any symmetry breaking in $M$. 
The derivative $d_h$ honors all Euclidean symmetries on $\mathbb{R}^q$. 
Symmetries are important because they are related to {\bf conservation
laws}. This {\bf N\"other relation} links momentum or angular momentum conservation
with translation and rotational symmetry for example. Any triangulation or
lattice approach would break symmetries. 

\paragraph{}
If $G$ is a symmetry group on $M$, it has a representation as unitary operators
in $\mathcal{H}$. If $T$ is a rotation for example, then $f \to f(T)f$ defines
a  unitary $U_T$ in $\mathcal{H}$ that commutes with $d$. It still commutes with 
the new derivative  $d_h$. 

\begin{thm}[Symmetry] 
If a unitary operator on $\mathcal{H}$ commutes with $d$, it continues to 
commute with $d_h$.
\end{thm}

\paragraph{}
The fact that symmetry is preserved and allows to transition easily from a continuum to a 
quantized situation. The hydrogen operator $L=-h^2 \Delta/2m-\frac{e^2}{4 \pi \epsilon_0}/r$ 
for example on $0$-forms has eigenvalues $\lambda_n = -C/n^2$ with 
$C=m e^4/(2 (4\pi \epsilon_0)^2 h^2)$. 
Now $d_h \sim d h$ and $L_h \sim L h^2 = -\Delta h^2$. 
The Hydrogen operator in spherical coordinates 
is now $O(h^4)$ close to the deformed operator $L_h/m - \frac{e^2}{4 \pi \epsilon_0}/r$. 
The eigenvalues in the deformed situation are $h^2$ close.
In the case of the Hydrogen atom, one has 
$C=mc^2 \alpha^2/2 = 13.605 ... eV =2.1798 ...\cdot 10^{-18} J$, 
with the {\bf fine-structure constant} $\alpha \sim 1/137$.
This $C$ is the {\bf ground state energy}. When replacing $L$ with $L_h$, the spectrum is 
indistinguishable. Modeling the Hydrogen atom using finite geometries would be hard because 
spherical coordinates are not available in such a frame work.  
If we replace $-h^2 \Delta$ with the zero Laplacian in $L_h$ and take $h$ to be of 
Planck scale, all the physics would look the same. 

\paragraph{}
The {\bf Maxwell equations} can be played on any Riemannian manifold 
\footnote{pseudo Riemannian manifold} $M$ with exterior derivative $d$.
Given a $1$-form $j$, it is the system of partial differential equation $dF=0, d^* F = j$
for an unknown $2$-form $F$. If $F=dA$ (which is possible if $M$ is simply connected or even if $b_1=0$),
then Maxwell means $d^* dA=j$. Under a {\bf Coulomb gauge} $A \to A+df$, we can assume $d^*A=0$
which means that the Maxwell equations are reduced to the {\bf Poisson equation} 
$L A = (d d^* + d^* d) A = j$. 
If $j$ is in the orthogonal complement of the harmonic forms ${\rm ker}(L)$, we have a unique
solution $A=L^{-1} j$, (where $L^{-1}$ is the pseudo inverse). 
The electro-magnetic field then is given as $F=dA$. In 
vacuum and in a space time manifold with signature $-+++$, the equation $L u=0$ is a wave equation. 
In our case, where time is separated from space, we have the d'Alembert operator 
$\Box{}=\partial_t^2 + L$ so that the
wave equation in space dimension $q$ is $A u =0$. The deformed d'Alembert operator is 
$\Box{}=D_{tt} + L$. We have seen that $\Box{} u =0$ with $u=0, \lim_{t \to 0} u_t(t,x) = du(x)$ 
is solved by $u(t,x) = d_t f(x) = t \phi_{q+2}(tD) df$. 

\section{Questions}

\paragraph{}
A discrete time wave evolution $T: (u,v) \to (D_h u-v,u)$ is defined, if $h$ is such that
$D_h$ has norm smaller than $1$. It is conjugated to $(e^{i \tilde{D}_h},e^{-i \tilde{D}_h})$
for a slightly changed $\tilde{D}$. 
See \cite{Kni98}. For small $h$, this is almost indistinguishable from the actual unitary evolution.
What are the properties of this wave evolution? Do we still have the strong Huygens principle?
We liked the discrete cellular automaton type time evolution for numerical purposes so that the
wave or Schr\"odinger evolution in the discrete has finite speed of propagation. On a discrete
geometry with continuous time $t$, there is no locality for the wave equation $u_{tt} = -L u$;
we need also time to be discrete. The above evolution $T$ has a finite speed of propagation 
because $d_h$ only taps into parts of space in distance $h$. 

\paragraph{}
The spectrum of the deformed Laplacian $L_h= d_h^* d_h + d_h d_h^* = D_h^2$ is determined
by the spectrum of $L$ if $h>0$. (We write here $h$ rather than $t$ because one of 
the applications we have in mind is to use $d_h$ for small $h$, where $h$ is on the order of the 
{\bf Planck constant}. As $d_h^* = h \phi_{q+2}(h D) d^*$, we have 
$D_h = d_h+d_h^* = \phi_{q+2}(h D) hD$.
This implies that $D_h$ has the same eigenfunctions than $D$ and $\lambda_j(D_h) = \psi(\lambda_j(D))$,
where $\psi(r) = \phi_{q+2}(r) r$, a bounded function. 
Similarly, also the eigenfunctions of $L_h$ are the same than
the eigenfunctions of $L$ and $\lambda_j(D_h) = \psi(\lambda_j(L))$. 
It would be nice to use $L_h$ as a numerical tool. 
Is there a natural modification $V_h(r)$ of the potential $V(r) = 1/r$ so that we have a bounded
analog of the Hydrogen operator in the form of $L_h +V_h$?

\paragraph{}
The geometry of wave fronts can become very complicated, even in integrable situations like
the ellipsoid. Can we relate the spectrum of $L$ and so the spectrum of $d_t$ 
with the wave front? It would be nice for example to see spectrally 
whether wave front become dense. In the case of a round sphere, the wave front oscillates
and stays bounded. 

\paragraph{}
The classical wave equation $u_{tt} + L u=0$ has {\bf Klein-Gordon generalizations}
$u_{tt} + (L+m) u = 0$, which does not satisfy the strong Huygens principle. 
What happens in the modified case 
$u_{tt} + (q-1) u_t/t + (L+m) u$ or 
$u_{tt} + (q-1) u_t/t + (q-1) u/t^2 + (L+m) u$? Are there cases,
where we have sharp wave fronts? 

\paragraph{}
If we look at a 4-manifold and consider the front $W_t(p)$ to be space, we have a 
{\bf space-time model} which locally shows Lorentz symmetry. 
Moving $p$ around in $M$ is a space time change. 
Every starting point $p$ defines a different space-time structure.
On a Riemannian manifold, the radius of injectivity in general depends on $p$. 
So, there are points $p_1$, where the {\bf space 3-manifold} $W_t(p_1)$ is regular while 
for $p_2$ the {\bf space 3-manifold} $W_t(p_2)$ has developed a singularity as $t$ is larger
than the radius of injectivity. The space-time structure depends on the starting point $p$,
where the space-3-manifold is a point. On a 4-dimensional ellipsoid as space time 4-manifold,
there are initial points for which there is a ``big crunch", a time where $W_t(p)$ is a again
a point. For most other points, the space manifold expands indefinitely. 
This is highly non-trivial even in toy models like circular billiard, a 2-manifold with 
boundary where the length of the wave front of a "a drop of
water in a cup" asymptotically behaves like $|W_P(t)| = \arcsin(|P|) t + o(t)$ \cite{Vicente2020}.

\paragraph{}
For a finite abstract simplicial complex $G$ with $n$ elements,
there is an exterior derivative $d$ on the finite dimensional Hilbert space $l^2(G)$. 
Functions on $G_k= \{ x \in G, |x|={k-1} \}$ are $k$-forms. 
The exterior derivative maps $l^2(G_k)$ to $l^2(G_{k+1})$. If $x$ is a simplex, 
then $df(x)$ is the flux of $f$ through the boundary $\delta x$. We do not need
to deform $d$ because it is already a bounded operator. 
If $G$ is a $q$-manifold and $D=d+d^*$ is the Dirac matrix, a $n \times n$ matrix, 
then we can still do a deformation $D_t = \phi_{q+2}(D t) t D$. We can also look 
at the now ordinary differential equation $u_{tt} + (q-1)[ u_t/t - u/t^2] = -D^2 u$. 
We can still deform exterior derivatives using the geodesic flow, but what happens
there is still unexplored. 

\section{Conclusion} 

\paragraph{}
We have deformed the exterior derivative $d$ so that 
$d_h$ has become a bounded operator on $\mathcal{H}$. Boundedness
was important. Operators like $D$ or $L=D^2$ were only 
densely defined self-adjoint operators on the Hilbert space $\mathcal{H}$.
Having a bounded operator  implies that the operator
naturally extends to all elements of the Hilbert space $\mathcal{H}$. 
The new Laplacian $L_h = (d_h+d_h^*)^2$ behaves like $h^2 L$ for small $h$.

\paragraph{}
The {\bf Stokes theorem} $\int_G dF = \int_{\delta G} F$ for 
smooth $k$-forms in a compact $(k+1)$ manifold $G$ with 
$k$-manifold $\delta G$ as boundary now reads as
$$  \langle G, d_t F \rangle = \langle d_t^* G, F \rangle \; . $$
{\bf Fields} $F$ and smooth geometries $G$ (which are examples of 
{\bf de Rham currents}) are now on the same footing as both are 
elements in the same Hilbert space. There is a symmetry 
between fields and geometries. 

\paragraph{}
It follows for example that we can relax about the regularity of 
geometries. We have a bounded exterior derivative on manifolds
which are singular like polyhedra or varieties on which there is a natural 
geodesic metric. Geometries can also be fractal. If $G$ for example is the
{\bf Koch fractal}, a region $G$ in $\mathbb{R}^2$  with the fractal boundary $\delta G$
called Koch curve. If $F$ is a $1$-form, we can look at the curl $d_t F$ for some small $t$ 
and integrate this over $G$. It is just an inner product $\langle G, d_t F \rangle$
as we can see also write $G = 1_G(x,y) dxdy$ as a $2$-form. Now
this is $\langle d_t^* G, F \rangle$ and $d_t^* G$ is an $L^2$ 1-form a relatively
regular de-Rham current, compared with the traditional boundary. 

\paragraph{}
This derivative {\bf quantizes distances}: only $f(q)$ for $q$ in distance $h$ from 
$p$ matter when evaluating $d_h f(p)$. We have a {\bf calculus without limits} because
we do not need to take the limit $h \to 0$. In physical circumstances, taking a Planck
distance size $h$ produces the same result in traditional physics. 
The derivative in principle works for any Riemannian manifold for which singularities
in the metric are of codimension $2$. On a solid polyedron for example, a 3-manifold
with boundary, the calculus works as the singularities that affect the wave fronts are
locate on vertices only.

\paragraph{}
The calculus resembles the discrete case if $M$ is an abstract simplicial 
complex, where $df(p)$ for a $k$-simplex $p$ only needs to know $q$, 
where $q$ are the {\bf boundary simplices} of $p$ of codimension $1$. 
We can in discrete manifolds for example define a new exterior derivative 
that instead of summing up over all boundary simplices sums up over all 
boundary simplices of all adjacent simplices. 

\section*{Appendix: Bessel }

\paragraph{}
The Bessel type differential equation $u'' + (q-1) u'/r  + u =0$ for a function 
$u(r)$ is the {\bf radial Helmholtz equation} for $k=1$ in dimension $q$, with energy $1$ and 
angular momentum $0$. We wrote $\phi_q(r)$ for its solution given the initial condition 
$u(0)=1, u'(0)=0$. The eigenvalue problem $\Delta u + c^2 u = 0$ becomes in 
spherical coordinates 
$$  u_{rr} + \frac{(q-1)}{r} u_r + \frac{1}{r^2} \Delta_{S^{q-1}} u + c^2 u = 0  \; . $$
If $r$ is replaced by $r/c$, the energy can be scaled to become $1$. 
With a substitution $\nu=(q-2)/2$ and $y(r)=r^{\nu} u(r)$ it becomes the 
{\bf standard Bessel equation}
$$ y'' + \frac{y'}{r} - \nu^2 \frac{y}{r^2}  + y = 0  $$       
The solutions are Bessel functions of order $\nu$ given by 
$r^\nu J_{\nu}(r)$ and  $r^\nu Y_{\nu}(r)$, which are known as 
{\bf Bessel functions of the first kind and second kind}.

\paragraph{}
Related to the deformed wave equation, we had introduced the
{\bf modified acceleration operator} $u'' + (q-1) [\frac{u'}{r} - \frac{u}{r^2}]$. 
It is the {\bf radial Helmholtz operator} 
$$ u'' + (q-1) \frac{u'}{r} - l(l+q-2) \frac{u}{r^2}  + c^2 u= 0 $$
with {\bf angular momentum} $l=1$ and {\bf energy} $c=0$ in {\bf dimension} $q$.

\paragraph{}
For angular momentum $l=1$ and energy $c=0$, this radial Helmholtz operator becomes
$$ u'' + (q-1) \frac{u'}{r} - (q-1) \frac{u}{r^2}   \; . $$
For $l=0$  and $c=0$, it is the operator
$$ u'' + (q-1) \frac{u'}{r}   \; .  $$
These are the time acceleration operators we have used in the modified wave equations. 

\paragraph{}
The equation
$$ u'' + (q-1) \frac{u'}{r} - (q-1) \frac{u}{r^2} + c^2 u =0  $$
becomes after a substitution $y=r^{((q-2)/2)} u$ the Bessel equation
$$ y'' + \frac{y'}{r} - (\frac{q}{2})^2 \frac{y}{r^2}  + c^2 y = 0  \; .  $$

\paragraph{}
The solutions $\phi_q$ to the Bessel equation 
$$ u''(r) + (n-1) \frac{u'(r)}{r} + u(r)=0, u(0)=1,u'(0)=0 $$
are also {\bf hypergeometric functions}. One can see this from the Taylor expansion 
$\phi_q(r) = \sum_{k=0}^{\infty} b_k r^{2k}$ with 
$1/b_k = B(q,k) = \prod_{j=1}^k (-2 j) (q - 2 + 2 j)$. 

\paragraph{}
To summarize, we have now seen four different ways to look at $\phi_q(r)$. It is a 
solution of a radial Helmholtz equation, expressible using Bessel functions of the 
first kind, given as a confluent hypergeometric functions or given in a series expansion which 
appeared, when looking at sphere averages. 

\begin{lemma}
$\phi_q(r) = F_1({\frac{q}{2},-\frac{r^2}{4}}) 
           = J_{\frac{q}{2}-1} \Gamma(\frac{q}{2}) (\frac{r}{2})^{\frac{q}{2}-1} 
           = \sum_{k=0}^{\infty} b_k r^{2k}$.
\end{lemma}

\paragraph{}
While already explained by changes of variables, it is best illustrated in 
the language of a computer algebra system which has these functions already baked in: 

\begin{tiny} 
\lstset{language=Mathematica} \lstset{frameround=fttt}
\begin{lstlisting}[frame=single]
B[q_,k_] := Product[(-2 j) (q-2+2 j),{j,1,k}];
CheckBesselIdentities[q_]:=Module[{},
phi    = f[r] /. First[DSolve[{f''[r]+(q-1)*f'[r]/r+f[r]==0,f[0]==1,f'[0]==0},f[r],r]];
hyper  = Hypergeometric0F1[q/2,-r^2/4]; 
bessel = BesselJ[q/2-1,r] Gamma[q/2]/(r/2)^(q/2-1);
series = Sum[r^(2 k)/B[q, k], {k, 0, Infinity}];
Print[{phi,hyper,bessel,series}];
FullSimplify[phi==hyper==bessel==series]];    
Table[CheckBesselIdentities[q],{q,1,10}]
\end{lstlisting} 
\end{tiny}

\paragraph{}
From the classical Bessel identity $(r^n*J_n(r))'=r^n*J_{n-1}(r)$, one
gets that the functions $\phi_q$ satisfy the following recursion:

\begin{lemma}
$(\phi_{q+2}(r) r^q)' = q \phi_q(r) r^{q-1}$. 
\end{lemma} 

\paragraph{}
This shows that if $d_t f = t\phi_{q+2}(tD) df$ satisfies the strong Huygens property, 
also $\phi_q(tD) df$ satisfies the strong Huygens property. 
Since $\phi_{q+2}(tD) df$ only depends on $f \in W_t(p)$, we have for $h<t$
that $\phi_{q+2}(tD) df$ does not involve $f(x)$ with $x \in B_h(p)$. 
So, also the traditional time derivative $D \phi'_{q+2}(tD) df$ does not depend on $x \in B_h(p)$
and also $D \phi'_{q+2}(tD) (tD)^{1-q} df$ not, where $D^{-1}$ is understood as the pseudo inverse.
But the above identity sees this as $D q \phi_q(tD) df$. So, also $\phi_q(tD) df$ does
not depend on $x \in B_h(p)$.

\section*{Appendix: Kirchhoff equations}

\paragraph{}
Reusing the constants $B(q,k) = \prod_{j=1}^k (-2j) (q+2j)$,
\cite{Ovall2014} gives the following {\bf Taylor formula} for the mean value over a ball
$$ {\rm E}_{B_t}[g] = \sum_{k=0}^{\infty} \frac{t^{2k}}{B(q+2,k)} L^k  g \; $$
and over a wave front $W_t$ and
$$ {\rm E}_{W_t}[g] = \sum_{k=0}^{\infty} \frac{t^{2k}}{B(q,k)} L^k g \; .  $$
\footnote{\cite{Ovall2014} works with scalar Laplacians $\Delta=-L_0$ and
scalar functions $g \in C^{2p}$ and finite Taylor sums. We only need the
formula for polynomials, as they are dense in $L^2$. In general, polynomial
differential forms are dense in the Hilbert space $\mathcal{H}$. }

\paragraph{}
If $g=df$, where $f$ is a $(q-1)$-form, then ${\rm E}_{B_t}[g]$ is by {\bf Stokes theorem}
equal to the normalized flux $(1/|B_t|) \int_{W_t} f$ through the sphere $W_t$. 
Hence, the sphere and ball volumes relate as $|W_t| t=q |B_t|$ we have $(t/q)/|B_t| = 1/|W_t|$ 
and so ${\rm E}_{W_t}[f] = (t/q) {\rm E}_{B_t}[g]$. 

\paragraph{}
We observed that ${\rm E}_{B_t}[g] = \phi_{q+2}(D t) g$, and
that ${\rm E}_{W_t}[g] = \phi_{q}(D t) g$, where $\phi_q$ solves the 
{\bf Bessel differential equation}
$f_{rr} + (q-1) f_r/r + f=0$ with $f(0)=1, f_r(0)=0$. 

\paragraph{}
In the following, we could assume $f$ to be a (q-1) form in $\mathcal{H}_{q-1}$ but then it would only 
hold for almost all $p$. If $f$ is assumed to be continuous, then it holds for all center points $p$. 

\begin{coro}
If $f$ is a continuous $(q-1)$-form, then 
$d_t f(p) = t \phi_{q+2}(Dt) df(p)$ is $q$ times the average flux of 
$f$ through $W_t(p)$. 
\end{coro}

\paragraph{}
This illustrates the strong Huygens principle for the deformed wave equation. But this 
interpretation as an average flux only holds in Euclidean space: we have used a 
relation between ball and sphere volumes which do not hold in a general Riemannian manifold. 
In a general Riemannian manifold $M$, it is still a multiple of the flux and we still have a 
strong Huygens principle. 

\paragraph{}
The {\bf Kirchhoff solutions} of the classical wave equations generalizes to all dimensions 
(see \cite{Evans2010} page 73 or \cite{Simon2017} Theorem 6.9.8, Theorem 6.9.10). 
One in general assumes the wave equation to evolve scalar fields $g$.
It also applies to $q$-forms = volume forms $g dx$, which is the case we use. 

\paragraph{}
Define ${\rm E}_{W_t}[g]$ and ${\rm E}_{B_t}[g]$ as averages 
over the sphere or ball. Define the pseudo differential operators
$T_t=[(\frac{1}{t} \partial_t)]^{1/2}$ and $\gamma_q = 1,3,5 \cdots (q-2)$. Assume the
initial condition $g$ and initial velocity $h$ are given. In odd dimensions:
$$ u = \frac{1}{\gamma_q} \partial_t T_t^{(q-3)/2} (t^{q-2} {\rm E}_{W_t}[g]) 
     + \frac{1}{\gamma_q}            T_t^{(q-3)/2} (t^{q-2} {\rm E}_{W_t}[h]) \; . $$
Define in $B_t(x)$ the function $\tilde{g}(y)= g(y)/\sqrt{t^2-|x-y|^2}$.
In even dimensions,
$$ u = \frac{1}{\gamma_q} \partial_t T_t^{(q-2)/2} (t^{q} {\rm E}_{B_t}[\tilde{g}]) 
     + \frac{1}{\gamma_q}            T_t^{(q-2)/2} (t^{q} {\rm E}_{B_t}[\tilde{h}]) \; . $$
The {\bf strong Huygens principle} is that the solution $u(x,t)$ in odd dimensions
only depends on information of $g,h$ on the wave front $W_t$ and not the entire ball. 
The {\bf weak Huygens principle} is that the solution $u(x,t)$ in even dimensions
only depends on information of $g,h$ in the wave ball $B_t$, not outside. There is
{\bf finite speed of propagation}. 

\section*{Appendix: Polarization}

\paragraph{}
Given $n$ variables $x_1, \dots, x_n$ in an arbitrary commutative ring $\mathcal{R}$ with $1$. 
Define $S=\{-1,1\}^n$, the set of all $\pm 1$ strings of length $n$. 
The {\bf polarization identity}  is 

\begin{lemma}
$n! \prod_{i=1}^n x_i = |S|^{-1} \sum_{s \in S} (\prod_j s_j) (\sum_i s_i x_i)^n  \; .$
\end{lemma} 

\paragraph{}
It shows that every element in the polynomial ring $\mathcal{R}[x_1, \dots, x_n]$ 
can be written as a linear combination of powers of linear functions. 
This can then applied also to cases, where some $x_i$ are the same, allowing so to 
write any monomial $x_1^{m_1} x_2^{m_2} \cdots x_k^{m_k}$ as a linear combination of
powers $(\sum_{i=1}^k a_i x_i)^m$ of linear functions, where 
$m=m_1 +m_2 + \cdots + m_k$ is the degree of the monomial.

\paragraph{}
The best known case of the polarization identity is the {\bf parallelogram law}.
It writes $xy$ as $[(x+y)^2-(x-y)^2 -(-y+x)^2 + (-x-y)^2]/8$ where $8=2!*2^2$.
The parallelogram law is used to show that in any Hilbert space, 
we only need to know {\bf lengths} to get {\bf angles} as the {\bf dot product} 
in a Hilbert space is recoverable from {\bf norms}. It can be
used in statistics as ${\rm Cov}[X,Y]= ({\rm Var}[X+Y] - {\rm Var}[X-Y])/4$. 
An other example is
$xyz = \frac{1}{3! 2^3} \sum_{s} s_1 s_2 s_3 (s_1 x+s_2 y+s_3 z)^3$. It can
also cover cases like $x^2 y$ and see it as a linear combination of cubic powers 
$(ax+by)^3$ of linear functions in $x$ and $y$. 

\paragraph{}
The proof of the polarization property is done by averaging over the set $S$ of all possible 
$\pm 1$ vectors $s$. By symmetry, it is a multiple of a monomial. The symmetrization renders 
all signs the same, so that it must be $C \prod_{i=1}^n x_i$ for some constant C. 
To find the constant, plug in $x_i=1$ and compute $\sum_{s \in S} (\prod_j s_j) (n-2k)^n$. 
It reduces to show that $R=\frac{1}{n!} \sum_{k=0}^n \binom{n}{k} (-1)^k (n-2k)^n = 2^n$.


\paragraph{}
Plugging in $(n-2k)^n=\sum_{j=0}^n \binom{n}{j}n^{\,n-j}(-2k)^j$ gives
$$ R=\frac1{n!}\sum_{j=0}^n \binom{n}{j}n^{\,n-j}(-2)^j \sum_{k=0}^n (-1)^k\binom{n}{k}k^j  \; . $$
The finite-difference identity show that $\sum_{k=0}^n (-1)^k \binom{n}{k} k^j=0$ is only nonzero 
for $j=n$, where it is $(-1)^n n!$. 

\section*{Appendix: computer algebra related to Theorem 1}

\paragraph{}
We list now some computer algebra that we had used to find the partial 
differential equation.
\footnote{It was pretty much trial and error, first mostly done by hand, 
going through a few failures at first.}
Since Mathematica has built in $D$ as the derivative operator, we use $d$ for the 
Dirac matrix $D$. Note that all the verification of Theorem 1 works in the Abelian operator
algebra generated by the operator $D$, so that $D$ can just be treated as a
variable. The two computations are almost identical. In the 
first case, we look at $u(t,x) = t \phi_{q+2}(tD) Y$ with $Y=df$, 
and verify, that it satisfies $u_{tt} + (q-1)[u_t/t-u_{tt}/t^2]= -D^2 u$. 

\begin{tiny} 
\lstset{language=Mathematica} \lstset{frameround=fttt}
\begin{lstlisting}[frame=single]
q=5; g=First[f[r]/.DSolve[{f''[r]+(q+1)f'[r]/r+f[r]==0,f[0]==1,f'[0]==0},f[r],r]];
phi=g/.r->d t;  u=t phi d f;    {Limit[u, t -> 0],   Limit[D[u,t],t->0]}
FullSimplify[D[u,{t,2}]+(q-1)D[u,t]/t-(q-1)u/t^2 ==-d^2*u] 
\end{lstlisting}
\end{tiny}

\paragraph{}
In the second case, we look at $u(t,x) = \phi_{q}(t D) Y$ with $Y=df$. 
Note that now, there is no $t$ in front
of the solution $u(t,x)$ and that instead of $\phi_{q+2}$ we have now $\phi_q$. 
The curve $u(t,x)$ satisfies $u_{tt} + (q-1)[u_t/t]= -D^2 u$ with initial 
conditions $u(0,x) = df$ and $u_t(0,x)=0$.

\begin{tiny} 
\lstset{language=Mathematica} \lstset{frameround=fttt}
\begin{lstlisting}[frame=single]
q=7; g=First[f[r]/.DSolve[{f''[r]+(q-1)f'[r]/r+f[r]==0,f[0]==1,f'[0]==0},f[r],r]];
phi=g/.r->d t;  u=  phi d f;    {Limit[u, t -> 0],   Limit[D[u,t],t->0]}
FullSimplify[D[u,{t,2}]+(q-1)D[u,t]/t ==-d^2*u]
\end{lstlisting}
\end{tiny}

\section*{Appendix: computer algebra for sphere averaging}

\paragraph{}
As for the sphere average formulas, this also started first by experimenting by trial
and error during the winter of 2010/2011. We noticed then experimentally 
that Taylor coefficients of flux are related to Bessel functions but could not prove this yet.
The sphere and ball average formulas due to Ovall \cite{Ovall2014} 
(which in the case $d=3$ are known as Pizzetti type formulas) made things
clear in the winter of 2025/2026. 
What the code does in dimension q=1,2,3 is to take a random polynomial 
$(q-1)$-form and integrate the divergence over a ball of radius $r$, then 
make the Taylor expansion in $r$, and then  notice
that this is related to coefficients appearing in Bessel functions. 

\begin{tiny}
\lstset{language=Mathematica} \lstset{frameround=fttt}
\begin{lstlisting}[frame=single]
q=1; R:=Random[Integer,5]; L[f_]:=D[f,{x,2}]; F=Sum[R x^n,{n,0,9}];
div=D[F,x];    X=Integrate[ div /. x->r,{r,-s,s}];
c=Simplify[CoefficientList[Series[X/(2s),{s,0,12}],s]];
c0=c[[1]]; c1=c[[3]]; c2=c[[5]]; c3=c[[7]]; c4=c[[9]];
a0=div /. x->0 /. y->0;             a1=L[div] /. x->0 /. y->0;
a2=L[L[div]] /. x->0 /. y->0;       a3=L[L[L[div]]] /. x->0 /. y->0;
a4=L[L[L[L[div]]]] /. x->0 /. y->0;
B[q_,m_]:=1/Product[-(-2j)(q+2j),{j,1,m}];
{c0,c1,c2,c3,c4}-{a0 B[1,0],a1 B[1,1],a2 B[1,2],a3 B[1,3],a4 B[1,4]}


q=2; R:=Random[Integer,5]; L[f_]:=D[f,{x,2}] + D[f,{y,2}];
F={P,Q}={Sum[R x^n y^m,{n,0,9},{m,0,9}],Sum[R x^n y^m,{n,0,9},{m,0,9}]};
div=D[P,x]+D[Q,y]; 
X=Integrate[ r div /. x->r Cos[t] /. y->r Sin[t],{t,0,2Pi},{r,0,s}];
c=Simplify[CoefficientList[Series[X/(s^2 Pi),{s,0,12}],s]];
c0=c[[1]]; c1=c[[3]]; c2=c[[5]]; c3=c[[7]]; c4=c[[9]];
a0=div /. x->0 /. y->0;             a1=L[div] /. x->0 /. y->0;
a2=L[L[div]] /. x->0 /. y->0;       a3=L[L[L[div]]] /. x->0 /. y->0;
a4=L[L[L[L[div]]]] /. x->0 /. y->0;
B[q_,m_]:=1/Product[-(-2j)(q+2j),{j,1,m}];
{c0,c1,c2,c3,c4}-{a0 B[2,0],a1 B[2,1],a2 B[2,2],a3 B[2,3],a4 B[2,4]}


q=3; RR:=Random[Integer,4]; RRR:=Sum[RR x^n y^m z^k,{n,0,6},{m,0,6},{k,0,6}];
L[f_]:=D[f,{x,2}] + D[f,{y,2}] + D[f,{z,2}]; F={P,Q,R}={RRR,RRR,RRR};
div = D[P, x] + D[Q, y] + D[R, z];
X=Integrate[r^2*Sin[s] div/.x->r*Sin[s]*Cos[t]/.y->r Sin[s] Sin[t]/.z->r Cos[s],
  {t,0,2Pi},{s,0,Pi},{r,0,w}];
c=Simplify[CoefficientList[Series[X*3/(4Pi w^3),{w,0,12}],w]];
c0=c[[1]]; c1=c[[3]]; c2=c[[5]]; c3=c[[7]]; c4=c[[9]];
a0=div /. x->0 /. y->0 /. z->0;             a1=L[div] /. x->0 /. y->0 /. z->0;
a2=L[L[div]] /. x->0 /. y->0 /. z->0;       a3=L[L[L[div]]] /. x->0 /. y->0 /. z->0;
a4=L[L[L[L[div]]]] /. x->0 /. y->0 /. z->0;
B[q_,m_]:=1/Product[-(-2j)(q+2j),{j,1,m}];
{c0,c1,c2,c3,c4} - {a0 B[3,0],a1 B[3,1],a2 B[3,2],a3 B[3,3],a4 B[3,4]}
\end{lstlisting}
\end{tiny}

\section*{Appendix: illustrations}

\paragraph{}
We made our first experiments in this project by computing line integrals 
on circles on round 2-spheres. This was in the context of teaching multi-variable
calculus and exploring new problems for exams. We switch to multi-variable calculus
language: take an arbitrary vector field 
$F=[P,Q,R]$ and compute the line integrals $\int_C F \; dr$ along geodesic 
spheres C of geodesic radius $r$ on the 2-sphere. Then 
integrate the result over $M$, which for polynomial $F$ is an exercise in 
double integrals, alas a bit too complicated in general to do by hand. 
We noticed first experimentally that the result was zero, then explained it 
using cancellation. 

\begin{figure}[!htpb]
\scalebox{0.45}{\includegraphics{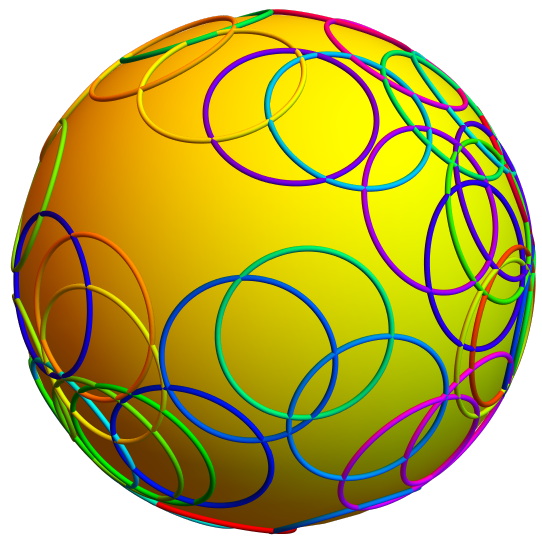}}
\label{Sphere}
\caption{
A non-local exterior derivative of a $1$-form $f$ on a 2-manifold evaluates 
a line integral of $f$ along a wave front. By Green's theorem applied to a
coordinate patch, it is an integral of curl.
}
\end{figure}

\paragraph{}
In the Riemannian geometry case the metric for which the metric lacks symmetry, 
the computation of the wave fronts requires to solve the geodesic 
differential equations $\ddot{x} + \Gamma_{ij}^k \dot{x}^i \dot{x}^j = 0$. 
The length of the wave front $W_t(p)$ starting at $p \in M$ is 
$$  |W_t(p)| = \int_0^{2\pi} |J(t,\theta)| \; d\theta \; , $$
where $J(t,\theta)$ is the {\bf Jacobi field} along the geodesic $x(t)$ 
starting at $p$ in direction $\theta$. Each $J(t,\theta)$ satisfies the Jacobi
differential equations
$$ J_{tt}(t,\theta)+K(x(t)) J(t,\theta)=0, J(0,\theta)=0,J_t(0,\theta)=1 \;, $$
where $K(x(t))$ is the curvature. See \cite{BergerPanorama}. In regions of 
negative curvature, the wave front length expands fast.

\begin{figure}[!htpb]
\scalebox{0.55}{\includegraphics{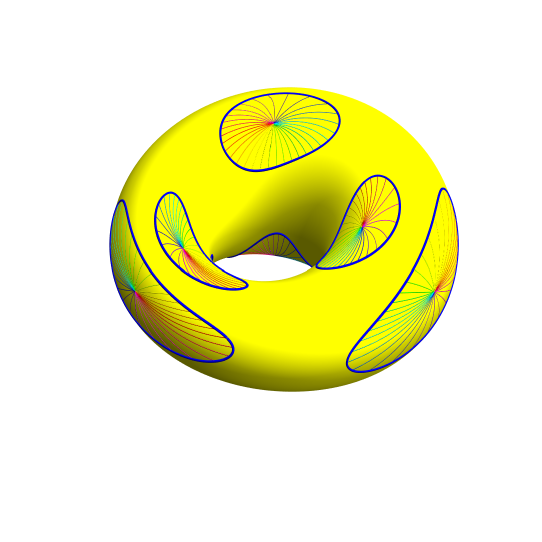}}
\label{Torus}
\caption{
On a general Riemannian 2-manifold we can evaluate a line integral along
a wave front $W_t(p)$ and get so a notion of a {\bf discrete derivative} for $1$-forms. 
By Green's theorem, it is a double integral over the ball $B_t(p) = \{ x \in M, d(x,p) = t \}$
at least if $t$ is smaller than the radius of injectivity. 
}
\end{figure}

\paragraph{} 
Let $M$ be a compact 2-manifold with boundary. Given a 1-form $f$, define the line integral
along $C=W_t(p)$ with $p$ at the boundary 
to be $\int_C d_r f$, a line integral along a half circle.
The line integral part is supported in a neighborhood of the boundary.

\paragraph{}
Symmetry shows that the integral $\int_M d_t f = 0$ for any $(q-1)$ form $f$:
every point $p$ and for almost all initial direction $\theta$ has for a given
$t$ a wave front point $x=W_t(p,\theta)$ a unique geodesic. Continue the geodesic for
an other time interval $t$ to get a dual point $p'$ and new direction 
$\theta'$ in the coordinate system at $p'$ that reaches $x$ in time $t$. 
The points $p$ and $p'$ both have distance $t$ to $x$. 
The line integral contributions of $W_t(p,\theta)$ and $W_t(p',\theta')$ 
cancel. 

\begin{figure}[!htpb]
\scalebox{0.3}{\includegraphics{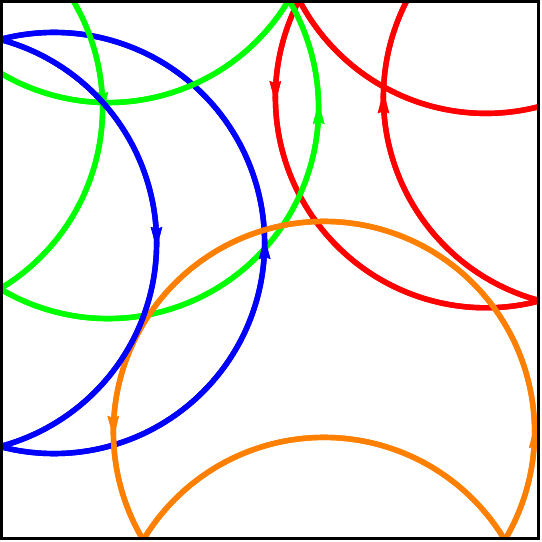}}
\label{Dirac}
\caption{
In a manifold with boundary, the wave front are reflected at the boundary. 
We still can show that the total integral of $df$ is zero. }
\end{figure}

\section*{Appendix: curvature via wave fronts}

\paragraph{}
We originally slithered into this topic in the context of {\bf Puiseux type formulas}
for curvatures in the discrete \cite{elemente11}. The point there had been 
{\bf not to take limits} but to use second order derivative notions. 

\paragraph{}
The starting point was the {\bf Bertrand-Diquest-Puiseux formula}
$K = \lim_{r \to 0} \frac{3 (2\pi r - |W_r|)}{r^3 \pi}$. 
It motivates to define 
$$ K_h(p) = \frac{2 |W_{h}| - |W_{2h}|}{2 \pi h^3} \; . $$
{\bf without taking limits} on any smooth surface $M$. 
We called this the {\bf R2-D2 formula} because the expression compares the length
of wave fronts for two radii in dimension two. 

\paragraph{}
For example, for a sphere of radius $R$, where $r = R \phi$ and
$|S_r| = 2\pi R \sin(r/R)$ the classical Puiseux formula gives $2(2\pi R \sin(r/R) - 2\pi r)$.
The R2-D2 formula gives
$4 \pi R \sin(r/R) - 2 \pi R \sin(2r/R) = 2 \pi r^3/R^2 - 2\pi r^5/(4 R^4) + ...$ so that $K = 1/R^2$.
Similarly, for a hyperbolic plane, where $|W_r| = 2 \pi R \sinh(r/R)$, we get $K = -1/R^2$.

\paragraph{}
At the boundary of a region, the R2-D2 formula is
$$ K = \lim_{r \to 0} \frac{ 2 |W_r| - |W_{2r}| } {2 \pi r^2} \; . $$
It is enough to verify this for a circular curve $C$ bounding a disc of radius $R$, where
$|W_r| = 2 \pi r \arccos(r/(2R)) = 2 \pi K r^2 + 7 \pi K^3 r^4/12 + ...$.

\paragraph{}
While the Puiseux formula refers to a flat situation, where the circle
has circumference $2\pi r$, the R2-D2 formula does refer to the
"flat" case. Flatness in this space emerged as the property that the length of 
a circular wave front grows linearly in time. Again, like here, the motivation for such 
definitions had been to have bounded curvature notions in situations which are not smooth, 
like polyhedra or fractal regions or even to push it to more general metric spaces. 
In general, it is necessary to adjust the constants to get Gauss-Bonnet. 

\begin{figure}[!htpb]
\scalebox{0.35}{\includegraphics{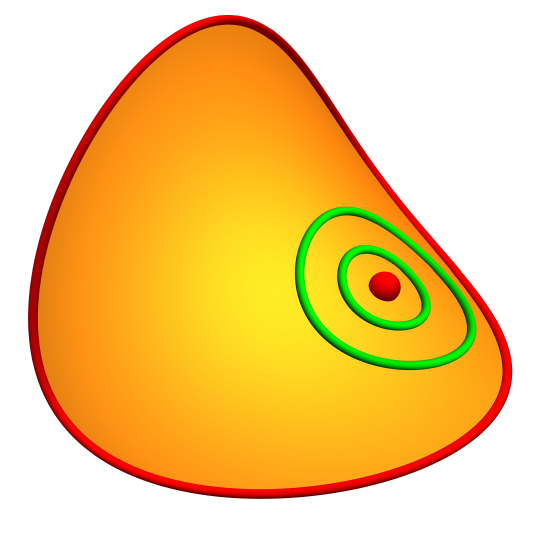}}
\scalebox{0.35}{\includegraphics{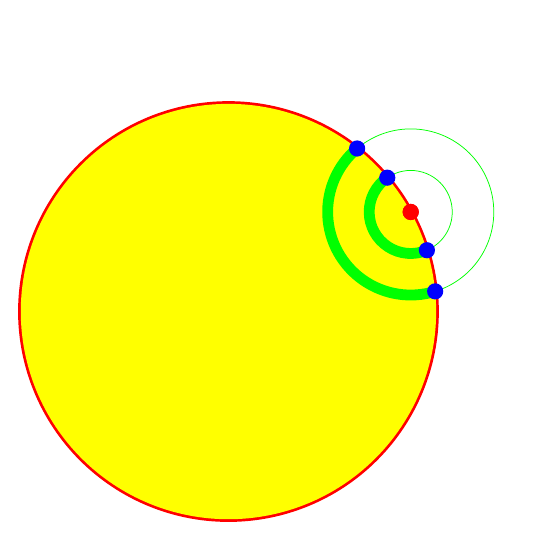}}
\label{Puiseux}
\caption{
R2-D2 formula for curvature for 2 manifolds in interior or boundary.
}
\end{figure}

\begin{figure}[!htpb]
\scalebox{0.35}{\includegraphics{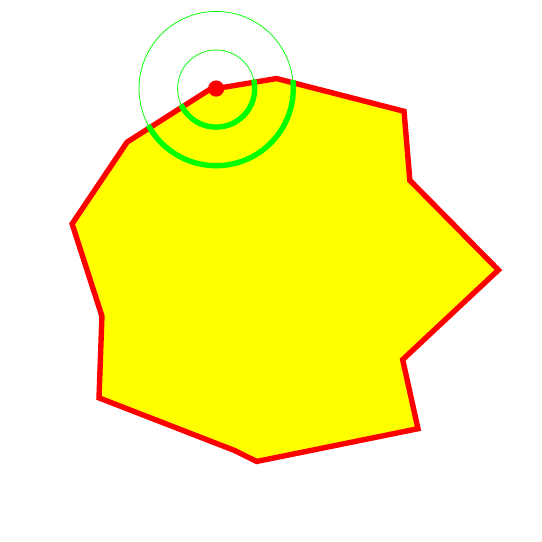}}
\scalebox{0.35}{\includegraphics{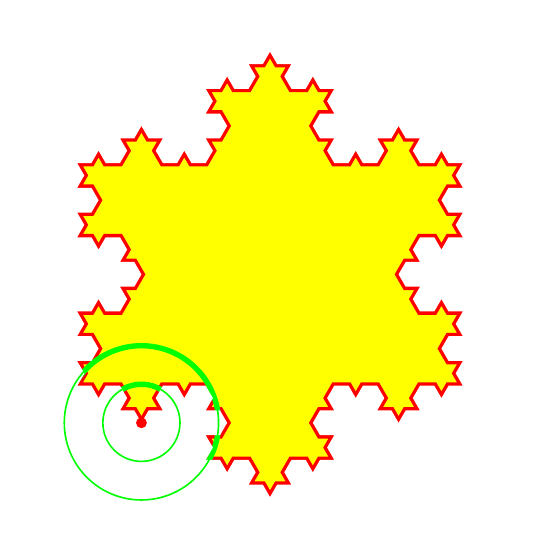}}
\label{Puiseux}
\caption{
R2-D2 formula at boundary of polygon or Koch region. 
}
\end{figure}

\bibliographystyle{plain}

\end{document}